\documentclass[a4paper, 10pt, twoside, notitlepage]{amsart}

\usepackage{amsmath,amscd}
\usepackage{amssymb}
\usepackage{amsthm}
\usepackage{comment}
\usepackage{graphicx, xcolor}

\usepackage{mathrsfs}
\usepackage[linktocpage,ocgcolorlinks, linkcolor=blue]{hyperref}

\usepackage{bm}
\usepackage{bbm}
\usepackage{url}

\usepackage[utf8]{inputenc}
\usepackage{mathtools,amssymb}
\usepackage{esint}
\usepackage{tikz}
\usepackage{dsfont}
\usepackage{relsize}
\usepackage{url}
\usepackage{xcolor}
\usepackage{graphicx}
\usepackage{mathrsfs}
\usepackage[shortlabels]{enumitem}
\usepackage{lineno}
\usepackage{amsmath}
\usepackage{enumitem}
\usepackage{amsthm} 
\usepackage{verbatim}
\usepackage{dsfont}
\numberwithin{equation}{section}

\allowdisplaybreaks

\mathtoolsset{showonlyrefs}

\graphicspath{{images/}}

\newtheorem{theorem}{Theorem}[section]
\newtheorem{lemma}[theorem]{Lemma}
\newtheorem{definition}[theorem]{Definition}

\newtheorem{proposition}[theorem]{Proposition}

\newtheorem{question}{Question}
\newtheorem{remark}[theorem]{Remark}

\title[Inverse problems for damped nonlocal wave equations]{Optimal Runge approximation for damped nonlocal wave equations and simultaneous determination results}

\author[P. Zimmermann]{Philipp Zimmermann}
\address{Departament de Matem\`atiques i Inform\`atica, Universitat de Barcelona, Barcelona, Spain}
\email{philipp.zimmermann@ub.edu}

\newcommand{\R}{{\mathbb R}}

\newcommand{\N}{{\mathbb N}}

\newcommand{\eps}{\varepsilon}

\newcommand{\schwartz}{\mathscr{S}}

\newcommand{\tempered}{\mathscr{S}^{\prime}}

\newcommand{\fourier}{\mathcal{F}}
\newcommand{\ifourier}{\mathcal{F}^{-1}}

\newcommand{\distr}{\mathscr{D}^{\prime}}

\newcommand{\abs}[1]{\left\lvert #1 \right\rvert}

\newcommand{\weakstar}{\overset{\ast}{\rightharpoonup}}

\begin{document}

	\maketitle
	\begin{abstract}
		The main purpose of this article is to establish new uniqueness results for Calder\'on type inverse problems related to damped nonlocal wave equations. To achieve this goal we extend the theory of very weak solutions to our setting, which allows to deduce an optimal Runge approximation theorem. With this result at our disposal, we can prove simultaneous determination results in the linear and semilinear regime.
		
		\medskip
		
		

	\end{abstract}

	\tableofcontents

	\section{Introduction}
	\label{sec: introduction}

     In recent years, inverse problems for nonlocal partial differential equations (PDEs) of elliptic, parabolic and hyperbolic type have been studied. This line of research was initiated by Ghosh, Salo and Uhlman \cite{GSU20}, in which they have considered the (partial data) Calder\'on problem related to the \emph{fractional Schr\"odinger equation}
	\begin{equation}
		\label{eq: fractional Schroedinger equation}
         \begin{cases}
        ((-\Delta)^s+q)u=0&\text{ in }\Omega,\\
        u =\varphi & \text{ in } \Omega_e,
    \end{cases}
	\end{equation}
	where $\Omega\subset\R^n$ is a bounded domain, $\Omega_e=\R^n\setminus \overline{\Omega}$, $0<s<1$, $q$ is a suitable potential and $(-\Delta)^s$ is the \emph{fractional Laplacian} which is the operator with Fourier symbol $|\xi|^{2s}$. In this problem one asks whether the knowledge of the \emph{(partial) Dirichlet to Neumann (DN) map}
	\begin{equation}
    \label{eq: partial DN map Schroeding}
		\Lambda_q \varphi= (-\Delta)^s u_\varphi|_{W_2}, \quad \varphi\in C_c^{\infty}(W_1),
	\end{equation}
    where $W_1,W_2\subset\Omega_e$ are given measurement sets (i.e.~nonempty open sets) and $u_\varphi$ denotes the unique solution to \eqref{eq: fractional Schroedinger equation}, uniquely determines the potential $q$. The overall strategy to establish unique determination results for the above Calder\'on problem is as follows (see \cite{GSU20,RS17,RZ-unbounded}):
    \begin{enumerate}[(i)]
        \item\label{item: integral identity} \emph{Integral identity:} Assume that the potentials $q_j$ are suitably regular, then one can write
        \begin{equation}
        \label{eq: integral identity schroeding}
            \langle (\Lambda_{q_1}-\Lambda_{q_2})\varphi_1,\varphi_2\rangle=\int_{\Omega} (q_1-q_2)(u_{\varphi_1}-\varphi_1),(u_{\varphi_2}-\varphi_2)\,dx,
        \end{equation}
        when the right hand side is interpreted accordingly.
        \item\label{item: Runge approx} Establish one of the following \emph{Runge approximation theorems}:
        \begin{enumerate}[(I)]
            \item\label{L2 Runge} $\mathcal{R}_W =\{u_f|_{\Omega}\,;\,f\in C_c^{\infty}(W)\}$ is dense in $L^2(\Omega)$ (see \cite{GSU20} for $q\in L^{\infty}(\Omega)$).
            \item\label{Hs Runge} $\mathscr{R}_W =\{u_f-f\,;\,f\in C_c^{\infty}(W)\}$ is dense in $\widetilde{H}^s(\Omega)$ (see \cite{RS17} for Sobolev multipliers $q$ or \cite{RZ-unbounded} for local, bounded bilinear forms).
        \end{enumerate}
        \item\label{item: Conclusion} If the potentials $q_j$ for $j=1,2$ have suitable continuity properties, then $\Lambda_{q_1}=\Lambda_{q_2}$ together with \ref{item: Runge approx} ensure that there holds $q_1=q_2$ in $\Omega$.
    \end{enumerate}
    In \ref{Hs Runge}, the space $\widetilde{H}^s(\Omega)$ is the closure of $C_c^{\infty}(\Omega)$ in the energy space
    \[
    H^s(\R^n)=\{u\in\tempered(\R^n)\,;\,\|u\|_{H^s(\R^n)}\vcentcolon = \|\langle D\rangle^s u\|_{L^2(\R^n)}<\infty\},
    \]
    where $\langle D\rangle^s$ is the Bessel potential operator. 
    Observe the similarity of the above strategy to the one of \cite{SU87} for showing unique determination for the classical Calder\'on problem, where instead of the Runge approximation theorem suitable geometric optics solutions are used. Moreover, the Runge approximation \ref{item: Runge approx} relies on a Hahn--Banach argument and the \emph{unique continuation property (UCP)} of the fractional Laplacian $(-\Delta)^s$. For more results on Calder\'on problems for elliptic nonlocal PDEs, we refer the interested reader to \cite{GLX,cekic2020calderon,CLL2017simultaneously,LL2020inverse,LL2022inverse, LZ2023unique,KLZ-2022,KLW2022,LRZ2022calder,Semilinear-nonlocal-wave-paper,LLU2023calder,CGRU2023reduction,LLU2023calder,RZ-unbounded,RZ2022LowReg,CRTZ-2022,LZ2024uniqueness,feizmohammadi2021fractional,feizmohammadi2021fractional_closed,FKU24,Trans-anisotropic-LNZ} and the references therein. 

    \subsection{Mathematical model and main results}
    \label{subsec: mathematical model and main results}

    Recently, the above approach for solving elliptic nonlocal inverse problems has also been adapted to deduce uniqueness results for the Calder\'on problem of nonlocal hyperbolic equations. Let us next describe some of these results in more detail and for this purpose consider the problem
     \begin{equation}
    \label{eq: discussion existing results}
        \begin{cases}
         \partial_t^2u+\lambda (-\Delta)^s \partial_t u+(-\Delta)^s u +f(u)= 0 & \text{ in } \Omega_T,\\
        u =\varphi & \text{ on } (\Omega_e)_T,\\
          u(0) = 0,\, \partial_{t}u(0) = 0 & \text{ on } \Omega,
    \end{cases}
    \end{equation}
    where $\lambda\in\R$ and $f\colon \Omega\times\R\to\R$ is a possibly nonlinear function. If the problem  \eqref{eq: discussion existing results} is well-posed in the energy class $H^s(\R^n)$, then for any two given measurement sets $W_1,W_2\subset\Omega_e$ we may introduce the DN map $\Lambda^{\lambda}_{f}$ via 
    \[
        \Lambda^{\lambda}_{f}\varphi=\left.(\lambda (-\Delta)^s\partial_t u_\varphi+(-\Delta)^s u_\varphi)\right|_{(W_2)_T},
    \] 
    whenever $\varphi$ is supported in $(W_1)_T$ and $u_\varphi$ is the solution of \eqref{eq: discussion existing results}. The \emph{Calder\'on problem} for \eqref{eq: discussion existing results} reads as follows:
    \begin{question}
    \label{question: calderon discussion}
        Does the DN map $\Lambda^{\lambda}_{f}$ uniquely determine the function $f$?
    \end{question}
    A suitable class of nonlinearities are the so-called weak nonlinearities, which are defined next.
    \begin{definition}\label{main assumptions on nonlinearities}
		We call a Carath\'eodory function $f\colon \Omega\times \R\to\R$ \emph{weak nonlinearity}, if it satisfies the following conditions:
		\begin{enumerate}[(i)]
			\item\label{prop f} $f$ has partial derivative $\partial_{\tau}f$, which is a Carath\'eodory function, and there exists $a\in L^p(\Omega)$ such that
			\begin{equation}
				\label{eq: bound on derivative}
				\left|\partial_\tau f(x,\tau)\right|\lesssim a(x)+|\tau|^r
			\end{equation}
			for all $\tau\in\R$ and a.e. $x\in\Omega$. Here the exponents $p$ and $r$ satisfy the restrictions 
            \begin{equation}
            \label{eq: restrictions on p}
                \begin{cases}
			n/s\leq p\leq \infty, &\, \text{if }\, 2s< n,\\
			2<p\leq \infty,  &\, \text{if }\, 2s= n,\\
			2\leq p\leq \infty, &\, \text{if }\, 2s\geq n.
		\end{cases}
            \end{equation}
			and
			\begin{equation}
				\label{eq: cond on r}
				\begin{cases}
					0\leq r<\infty, &\, \text{if }\, 2s\geq n,\\
					0\leq r\leq \frac{2s}{n-2s}, &\, \text{if }\, 2s< n,
				\end{cases}
			\end{equation}
			respectively. Moreover, $f$ fulfills the integrability condition $f(\cdot,0)\in L^2(\Omega)$.
			\item\label{prop F} The function $F\colon \Omega\times\R\to\R$ defined via
			\[
			F(x,\tau)=\int_0^\tau f(x,\rho)\,d\rho
			\]
			satisfies $F(x,\tau)\geq 0$ for all $\tau\in\R$ and $x\in\Omega$.
		\end{enumerate}
	\end{definition}
	Let us note that $f(x,\tau)=q(x)|\tau|^r\tau$ with $r$ satisfying \eqref{eq: cond on r} and $0\leq q\in L^{\infty}(\Omega)$ is a weak nonlinearity. Using the above notions, we can now discuss some of the existing results.
    \begin{enumerate}[(a)]
        \item\label{item: basic nonlocal wave} The article \cite{KLW2022} gives a positive answer for $\lambda=0, f(x,\tau)=q(x)\tau$ with $q\in L^{\infty}(\Omega)$. Their proof relied on the observation that the related nonlocal wave equation \eqref{eq: discussion existing results} satisfies an $L^2(\Omega_T)$ Runge approximation theorem. 
        \item\label{item: viscous nonlocal wave} The work \cite{zimmermann2024calderon} deals on the one hand with the linear case $\lambda=1$, $f(x,\tau)=q(x)\tau$ with $q\in L^{\infty}(0,T;L^p(\Omega))$, where $p$ satisfies the restrictions \eqref{eq: restrictions on p}, and $q$ is weakly continuous in $t$ and on the other hand with the nonlinear case $\lambda=1$ and $f$ is a $r+1$ homogeneous, weak nonlinearity. The uniqueness proofs use substantially that due to presence of the viscosity term $(-\Delta)^s\partial_t$ solutions $u$ to \eqref{eq: discussion existing results} satisfy $\partial_t u\in L^2(0,T;H^s(\R^n))$ and as a consequence the linearized equations have the Runge approximation property in $L^2(0,T;\widetilde{H}^s(\Omega))$.
        \item\label{item: semilinear nonlocal wave} In \cite{Semilinear-nonlocal-wave-eq}, uniqueness is proved in the case $\lambda=0$ and $f$ satisfies the same properties as in \ref{item: viscous nonlocal wave}, but with the additional restriction $r\leq 1$. This article only uses an $L^2(\Omega_T)$ Runge approximation result for the linearized nonlocal wave equation.
        \item\label{item: optimal runge} By establishing a theory for very weak solutions of linear nonlocal wave equations with $\lambda=0$, the authors of \cite{Optimal-Runge-nonlocal-wave} could deduce an optimal $L^2(0,T;\widetilde{H}^s(\Omega))$ Runge approximation theorem for these equations. This allowed to extend the results in \ref{item: semilinear nonlocal wave} to the cases $r>1$ and additionally showed that one can recover any linear perturbation $q\in L^p(\Omega)$ with $p$ satisfying the restrictions \eqref{eq: restrictions on p}. Furthermore, by this improved Runge approximation theorem the authors could also treat the case of serially or asymptotically polyhomogeneous nonlinearities (see \cite[Theorem 1.5]{Optimal-Runge-nonlocal-wave}). 
    \end{enumerate}
    In this context, let us also mention the recent article \cite{fu2024wellposednessinverseproblemsnonlocal} which deals with the Calder\'on problem for a third order semilinear, nonlocal, viscous wave equation. 

    The goal of this paper us to present an extension of the models described in \ref{item: viscous nonlocal wave} and \ref{item: semilinear nonlocal wave}, which we discuss next. Let $0<s<1$ and suppose that we have given coefficients $(\gamma,q)\in C^{0,\alpha}(\R^n)\times L^p(\Omega)$, where $0<s<\alpha\leq 1$ and $1\leq p\leq \infty$ satisfies the restrictions in  \eqref{eq: restrictions on p}.
    Then we define the following \emph{damped, nonlocal wave operator}
    \begin{equation}
    \label{eq: damped, nonlocal wave operator}
        L_{\gamma}\vcentcolon = \partial_t^2+\gamma\partial_t +(-\Delta)^s
    \end{equation}
    and consider the problem
    \begin{equation}
    \label{eq: wave problem}
        \begin{cases}
         L_{\gamma}u +f(u)= F & \text{ in } \Omega_T,\\
        u =\varphi & \text{ on } (\Omega_e)_T,\\
          u(0) = u_0,\, \partial_{t}u(0) = u_1 & \text{ on } \Omega,
    \end{cases}
    \end{equation}
    where $f$ is a weak nonlinearity or $f(u)=q(x)u$. In fact, this is a possibly nonlinear generalization of the model (G2) with $s_1=0$ in \cite[Section 1.1]{zimmermann2024calderon}. By \cite[Proposition 3.7]{Semilinear-nonlocal-wave-eq} (see Section \ref{subsec: weak solutions} for the linear case), we know that the problem \eqref{eq: wave problem} is well-posed, whenever the source $F$, exterior condition $\varphi$ and initial conditions $u_0,u_1$ are sufficiently regular. Thus, we can introduce the related (partial) DN map via
    \begin{equation}
    \label{eq: formal DN map}
        \Lambda_{\gamma,f}\varphi\vcentcolon = \left.(-\Delta)^s u_{\varphi}\right|_{(W_2)_T},
    \end{equation}
    where $W_1,W_2\subset \Omega_e$ are some measurement sets, $\varphi$ is supported in $(W_1)_T$ and $u_\varphi$ is the unique solution of \eqref{eq: wave problem} with $u_0=u_1=0$. Then we ask the following question: 

    \begin{question}
    \label{question: Caldeorn problem of this work}
        Does the partial DN map $\Lambda_{\gamma,f}$ uniquely determine the damping coefficient $\gamma$ and the function $f$?
    \end{question}

    In this work we establish the following affirmative answers to this question, whereas the first result discusses the linear case and the second one the semilinear perturbations.
   
    \begin{theorem}[Uniqueness for linear perturbations]
    \label{thm: uniqueness linear}
        Let $\Omega \subset\R^n$ be a bounded Lipschitz domain, $T>0$, $0<s<\alpha\leq 1$ and suppose that $1\leq p\leq \infty$ satisfies \eqref{eq: restrictions on p}. Assume that for $j=1,2$ we have given coefficients $(\gamma_j,q_j)\in C^{0,\alpha}(\R^n)\times L^{p}(\Omega)$ and let $\Lambda_{\gamma_j,q_j}$ be the DN map associated to the problem 
        \begin{equation}
        \label{eq: PDEs uniqueness theorem}
            \begin{cases}
         
         (L_{\gamma_j}+q_j)u = 0 & \text{ in } \Omega_T,\\
        u =\varphi & \text{ on } (\Omega_e)_T,\\
          u(0) = 0,\, \partial_{t}u(0) = 0 & \text{ on } \Omega
    \end{cases}
        \end{equation}
        for $j=1,2$. If $W_1,W_2\subset\Omega_e$ are two measurement sets such that
        \begin{equation}
        \label{eq: equality of DN maps}
            \left.\Lambda_{\gamma_1,q_1}\varphi\right|_{(W_2)_T}=\left.\Lambda_{\gamma_2,q_2}\varphi\right|_{(W_2)_T}
        \end{equation}
        for all $\varphi\in C_c^{\infty}((W_1)_T)$, then there holds
        \begin{equation}
        \label{eq: equal coefficients}
            \gamma_1=\gamma_2\text{ and }q_1=q_2 \text{ in }\Omega.
        \end{equation}
    \end{theorem}

    \begin{theorem}[Uniqueness for semilinear perturbations]
    \label{thm: uniqueness semilinear}
        Let $\Omega \subset\R^n$ be a bounded Lipschitz domain, $T>0$ and $0<s<\alpha\leq 1$. Assume that for $j=1,2$ we have given coefficients $\gamma_j\in C^{0,\alpha}(\R^n)$ and $r+1$ homogeneous, weak nonlinearities $f_j$, where $r>0$ satisfies \eqref{eq: cond on r}. Let $\Lambda_{\gamma_j,f_j}$ be the DN map associated to the problem 
        \begin{equation}
        \label{eq: PDEs uniqueness theorem}
            \begin{cases}
         L_{\gamma_j}u +f_j(u) = 0 & \text{ in } \Omega_T,\\
        u =\varphi & \text{ on } (\Omega_e)_T,\\
          u(0) = 0,\, \partial_{t}u(0) = 0 & \text{ on } \Omega
    \end{cases}
        \end{equation}
        for $j=1,2$. If $W_1,W_2\subset\Omega_e$ are two measurement sets such that
        \begin{equation}
        \label{eq: equality of DN maps semilinear}
            \left.\Lambda_{\gamma_1,f_1}\varphi\right|_{(W_2)_T}=\left.\Lambda_{\gamma_2,f_2}\varphi\right|_{(W_2)_T}
        \end{equation}
        for all $\varphi\in C_c^{\infty}((W_1)_T)$, then there holds
        \begin{equation}
        \label{eq: equal coefficients}
            \gamma_1=\gamma_2\text{ in }\Omega\text{ and }f_1=f_2 \text{ in }\Omega\times \R.
        \end{equation}
    \end{theorem}

    \begin{remark}
        For simplicity we restrict our attention to homogeneous nonlinearities $f$, but the unique determination remains valid in some polyhomogeneous cases as described in \cite{Optimal-Runge-nonlocal-wave} for $\gamma=0$.
    \end{remark}


    \subsection{Organization of the article} The rest of this article is structured as follows. In Section \ref{sec: Very weak solutions to damped, nonlocal wave equations}, we establish the existence of unique weak and very weak solutions to damped nonlocal wave equations. In Section \ref{sec: inverse problem} we then move on to the inverse problem part of this work. First, in Section \ref{sec: Runge approx} we establish the optimal Runge approximation theorem. Afterwards, in Section \ref{sec: integral identity} we prove a suitable integral identity that allows us to recover simultaneously the damping coefficient $\gamma$ and potential $q$ in Section \ref{sec: linear uniqueness}. Finally, Section \ref{sec: semilinear uniqueness} contains the proof of the simultaneous determination of the damping coefficient $\gamma$ and the homogeneous nonlinearity $f$.

	\section{Weak and very weak solutions to damped, nonlocal wave equations}
	\label{sec: Very weak solutions to damped, nonlocal wave equations}

    The main purpose of this section is to show existence of unique weak and very weak solutions to \emph{damped, nonlocal wave equations (DNWEQ)}
    \begin{equation}
    \label{eq: damped nonlocal wave equations well-posedness}
        \begin{cases}
         (L_{\gamma}+q)u = F & \text{ in } \Omega_T,\\
        u =\varphi & \text{ on } (\Omega_e)_T,\\
          u(0) = u_0,\, \partial_{t}u(0) = u_1 & \text{ on } \Omega,
    \end{cases}
    \end{equation}
    where $L_\gamma$ is given by \eqref{eq: damped, nonlocal wave operator} and only the case $\varphi=0$ is considered for very weak solutions. 

    \subsection{Weak solutions}
    \label{subsec: weak solutions}

    This section deals with the well-posedness of \eqref{eq: damped nonlocal wave equations well-posedness} for regular sources, exterior conditions and initial data. We also prove well-posedness for the case when instead of initial values the values at $t=T$ are specified, which will be needed for the development of the theory of very weak solutions.

    \begin{theorem}[Weak solutions to homogeneous DNWEQ]
    \label{thm: weak sol to hom DNWEQ}
       Let $\Omega \subset\R^n$ be a bounded Lipschitz domain, $T>0$, $0<s< 1$ and suppose that $1\leq p\leq \infty$ satisfies \eqref{eq: restrictions on p}. Assume that we have given coefficients $(\gamma,q)\in L^{\infty}(\Omega)\times L^{p}(\Omega)$. Then for any $F\in L^2(0,T;\widetilde{L}^2(\Omega))$\footnote{Here and below we set $\widetilde{L}^2(\Omega)\vcentcolon =\widetilde{H}^0(\Omega)$.} and initial conditions $(u_0,u_1)\in \widetilde{H}^s(\Omega)\times\widetilde{L}^2(\Omega)$, there exists a unique weak solution $u\in C([0,T];\widetilde{H}^s(\Omega))\cap C^1([0,T];\widetilde{L}^2(\Omega))$ of 
       \begin{equation}
    \label{eq: damped nonlocal wave equations well-posedness weak sol}
        \begin{cases}
         (L_{\gamma}+q)u = F & \text{ in } \Omega_T,\\
        u =0 & \text{ on } (\Omega_e)_T,\\
          u(0) = u_0,\, \partial_{t}u(0) = u_1 & \text{ on } \Omega,
    \end{cases}
    \end{equation}
    which means that $(u(0),\partial_tu(0))=(u_0,u_1)$ in $\widetilde{H}^s(\Omega)\times \widetilde{L}^2(\Omega)$ and there holds
    \begin{equation}
    \label{eq: weak formulation damped nonlocal wave eq}
    \begin{split}
        &\frac{d}{dt} \langle \partial_t u,v\rangle_{L^2(\Omega)}+\langle \gamma\partial_t u,v\rangle_{L^2(\Omega)}+\langle (-\Delta)^{s/2}u,(-\Delta)^{s/2}v\rangle_{L^2(\R^n)}+\langle qu,v\rangle_{L^2(\Omega)}\\
        &=\langle F,v\rangle_{L^2(\Omega)}
    \end{split}
    \end{equation}
      for all $v\in \widetilde{H}^s(\Omega)$ in the sense of distributions on $(0,T)$. Moreover, the unique solution $u$ obeys the energy identity
        \begin{equation}
        \label{eq: energy identity}
            \begin{split}
         &\|\partial_t u(t)\|_{L^2(\Omega)}^2+\|(-\Delta)^{s/2}u(t)\|_{L^2(\R^n)}^2+2 \int_0^t \langle \gamma\partial_t u(\tau)+qu(\tau),\partial_t u(\tau)\rangle_{L^2(\Omega)}\,d\tau\\
         &=\|(-\Delta)^{s/2}u_0\|_{L^2(\R^n)}^2+\|u_1\|_{L^2(\Omega)}^2+2 \int_0^t\langle F(\tau),\partial_tu (\tau)\rangle_{L^2(\Omega)}\,d\tau
    \end{split}
        \end{equation}
        which implies
        \begin{equation}
        \label{eq: energy estimate}
        \begin{split}
            &\|\partial_t u(t)\|_{L^2(\Omega)}+\|(-\Delta)^{s/2}u(t)\|_{L^2(\R^n)}\\
            &\leq C(\|u_1\|_{L^2(\Omega)}+\|(-\Delta)^{s/2}u_0\|_{L^2(\R^n)}+\|F\|_{L^2(0,t;L^2(\Omega))})
        \end{split}
        \end{equation}
        for all $0\leq t\leq T$ and some $C>0$ only depending on $\|q\|_{L^p(\Omega)}$, $\|\gamma\|_{L^{{\infty}(\Omega)}}$ and $T>0$. 
    \end{theorem}

    \begin{proof}
       Throughout the proof, we endow $\widetilde{H}^s(\Omega)$ with the equivalent norm  $\|u\|_{\widetilde{H}^s(\Omega)}=\|(-\Delta)^{s/2}u\|_{L^2(\R^n)}$ (see \cite[Lemma 2.3]{Semilinear-nonlocal-wave-eq}) and we introduce the following continuous sesquilinear forms
		\begin{equation}
			\label{eq: sesquilinear forms}
			a_0(u,v)=\langle (-\Delta)^{s/2}u,(-\Delta)^{s/2}v\rangle_{L^2(\R^n)},\quad a_1(u,v)= \langle qu,v\rangle_{L^2(\Omega)}
		\end{equation}
		for $u,v\in \widetilde{H}^s(\Omega)$ and
        \begin{equation}
            b(u,v)=\langle \gamma u,v\rangle_{L^2(\Omega)}
        \end{equation}
        for $u,v\in \widetilde{L}^2(\Omega)$. Next, recall that by \cite[eq.~(3.7)]{Semilinear-nonlocal-wave-eq} one has
        \begin{equation}
        \label{eq: L2 estimate potential}
            \|qu\|_{L^2(\Omega)}\leq C\|q\|_{L^p(\Omega)}\|u\|_{\widetilde{H}^s(\Omega)}
        \end{equation}
        for all $u\in \widetilde{H}^s(\Omega)$. It is not hard to see that we can invoke the existence and uniqueness results \cite[Chapter XVIII, \S 5, Theorem~3 \& 4]{DautrayLionsVol5} (see \cite[p.~571]{DautrayLionsVol5}), which ensure the existence of a unique, real-valued solution $u\in C([0,T];\widetilde{H}^s(\Omega))\cap C^1([0,T];\widetilde{L}^2(\Omega))$ to \eqref{eq: damped nonlocal wave equations well-posedness weak sol}. Furthermore, by \cite[Chapter XVIII, \S 5, Lemma~7]{DautrayLionsVol5} the solution $ u$ satisfies the following energy identity
    \begin{equation}
    \label{eq: energy identity proof}
    \begin{split}
         &\|\partial_t u(t)\|_{L^2(\Omega)}^2+\|(-\Delta)^{s/2}u(t)\|_{L^2(\R^n)}^2+2 \int_0^t \langle \gamma\partial_t u(\tau)+qu(\tau),\partial_t u(\tau)\rangle_{L^2(\Omega)}\,d\tau\\
         &=\|(-\Delta)^{s/2}u_0\|_{L^2(\R^n)}^2+\|u_1\|_{L^2(\Omega)}^2+2 \int_0^t\langle F(\tau),\partial_tu (\tau)\rangle_{L^2(\Omega)}\,d\tau
    \end{split}
    \end{equation}
    for $0\leq t\leq T$. Hence, we have shown the identity \eqref{eq: energy identity}. Let us define $\Psi \in C([0,T])$ by 
    \[
        \Psi(t)\vcentcolon = \|\partial_t u(t)\|_{L^2(\Omega)}^2+\|(-\Delta)^{s/2}u(t)\|_{L^2(\R^n)}^2
    \]
    for $0\leq t\leq T$. Using \eqref{eq: L2 estimate potential}, \eqref{eq: energy identity proof} and $\gamma\in L^{\infty}(\Omega)$, we get
    \[
    \begin{split}
        &\Psi(t)\leq \Psi(0)+\int_0^t \|F(\tau)\|_{L^2(\Omega)}^2\,d\tau+ C\int_0^t (1+\|\gamma\|_{L^{\infty}(\Omega)}+\|q\|^2_{L^p(\Omega)})\Psi(\tau)\,d\tau
    \end{split}
    \]
   and via Gronwall's inequality we deduce the energy estimate
    \begin{equation}
        \begin{split}
            &\Psi(t)\leq C(\Psi(0)+\|F\|_{L^2(0,t;L^2(\Omega))}^2)
        \end{split}
    \end{equation}
    for all $0\leq t\leq T$ and some $C>0$ only depending on $\|q\|_{L^p(\Omega)}$, $\|\gamma\|_{L^{\infty}(\Omega)}$ and $T>0$. This establishes the estimate \eqref{eq: energy estimate}. 
    \end{proof}

    As a consequence we have the following result:

    \begin{proposition}[Weak solutions to inhomogeneous DNWEQ]
    \label{prop: Weak solutions to inhomogeneous DNWEQ}
        Let $\Omega \subset\R^n$ be a bounded Lipschitz domain, $T>0$, $0<s< 1$ and suppose that $1\leq p\leq \infty$ satisfies \eqref{eq: restrictions on p}. Assume that we have given coefficients $(\gamma,q)\in L^{\infty}(\Omega)\times L^{p}(\Omega)$. Then for any $F\in L^2(0,T;\widetilde{L}^2(\Omega))$, exterior condition $\varphi\in C^2([0,T];H^{2s}(\R^n))$ and initial conditions $(u_0,u_1)\in H^s(\R^n)\times L^2(\R^n)$ satisfying the compatibility conditions $u_0-\varphi(0)\in \widetilde{H}^s(\Omega)$ and $u_1-\partial_t\varphi(0)\in \widetilde{L}^2(\Omega)$, there exists a unique weak solution $u\in C([0,T];H^s(\R^n))\cap C^1([0,T];L^2(\R^n))$ of 
       \begin{equation}
    \label{eq: damped nonlocal wave equations well-posedness weak sol inhom}
        \begin{cases}
         (L_{\gamma}+q)u = F & \text{ in } \Omega_T,\\
        u =\varphi & \text{ on } (\Omega_e)_T,\\
          u(0) = u_0,\, \partial_{t}u(0) = u_1 & \text{ on } \Omega,
    \end{cases}
    \end{equation}
    which means that $u$ satisfies \eqref{eq: weak formulation damped nonlocal wave eq}, the $(u(0),\partial_t u(0))=(u_0,u_1)$ in $H^s(\R^n)\times L^2(\R^n)$ and $u=\varphi$ in $(\Omega_e)_T$ means that $u(t)=\varphi(t)$ a.e. in $\Omega_e$ for any $0<t<T$. Furthermore, the following energy estimate holds
    \begin{equation}
    \label{eq: energy estimate inhomogeneous}
         \begin{split}
            &\|\partial_t u(t)\|_{L^2(\R^n)}+\|(-\Delta)^{s/2}u(t)\|_{L^2(\R^n)}\\
            &\leq C(\|u_1\|_{L^2(\R^n)}+\|(-\Delta)^{s/2}u_0\|_{L^2(\R^n)}+\|\varphi\|_{C^2([0,t];H^{2s}(\R^n))}+\|F\|_{L^2(0,t;L^2(\Omega))})
        \end{split}
    \end{equation}
    for any $0\leq t\leq T$. 
    \end{proposition}

    \begin{proof}
        Observe, under the current regularity assumptions and compatibility conditions, that $u$ solves \eqref{eq: damped nonlocal wave equations well-posedness weak sol inhom} if and only if $w\vcentcolon = u-\varphi$ solves
       \begin{equation}
        \label{eq: Usual cauchy homogeneous}
            \begin{cases}
         (L_{\gamma}+q)w = F-(L_\gamma+q) \varphi & \text{ in } \Omega_T,\\
        w =0 & \text{ on } (\Omega_e)_T,\\
          w(0) = u_0-\varphi(0),\, \partial_{t}w(0) = u_1-\partial_t\varphi(0) & \text{ on } \Omega.
    \end{cases}
        \end{equation}
        The only fact to keep in mind is that if $u\in C([0,T];H^s(\R^n))$, then the condition $u(t)=\varphi(t)$ a.e. in $\Omega_e$ is equivalent to $u(t)-\varphi(t)\in \widetilde{H}^s(\Omega)$ as $\Omega\subset\R^n$ is a bounded Lipschitz domain. So, the assertions of Propsition \ref{prop: Weak solutions to inhomogeneous DNWEQ} follow immediately from Theorem \ref{thm: weak sol to hom DNWEQ}.
    \end{proof}

     Next, let us define for any $g\in L^1_{loc}(V_T)$, $V\subset\R^n$ open, its \emph{time-reversal}
    \begin{equation}
    \label{eq: time reversal}
        g^\star(x,t)=g(x,T-t).
    \end{equation}
    Then, we have the following lemma.
   \begin{lemma}
   \label{lemma: time reversal of solution}
       Let $\Omega \subset\R^n$ be a bounded Lipschitz domain, $T>0$, $0<s< 1$ and suppose that $1\leq p\leq \infty$ satisfies \eqref{eq: restrictions on p}. Assume that we have given coefficients $(\gamma,q)\in L^{\infty}(\Omega)\times L^{p}(\Omega)$. Let $F\in L^2(0,T;\widetilde{L}^2(\Omega))$, $\varphi\in C^2([0,T];H^{2s}(\R^n))$ and $(u_0,u_1)\in H^s(\R^n)\times L^2(\R^n)$ satisfying the compatibility conditions $u_0-\varphi(0)\in \widetilde{H}^s(\Omega)$ and $u_1-\partial_t\varphi(0)\in \widetilde{L}^2(\Omega)$. Then $u$ solves
        \begin{equation}
        \label{eq: Usual cauchy}
            \begin{cases}
         (L_{\gamma}+q)u = F & \text{ in } \Omega_T,\\
        u =\varphi & \text{ on } (\Omega_e)_T,\\
          u(0) = u_0,\, \partial_{t}u(0) = u_1 & \text{ on } \Omega,
    \end{cases}
        \end{equation}
        if and only if $u^\star$ solves
        \begin{equation}
        \label{eq: time reversed cauchy}
            \begin{cases}
         (L_{-\gamma}+q)v = F^\star & \text{ in } \Omega_T,\\
        v =\varphi^\star & \text{ on } (\Omega_e)_T,\\
          v(T) = u_0,\, \partial_{t}v(T) = u_1 & \text{ on } \Omega.
    \end{cases}
        \end{equation}
        In particular, for any $F\in L^2(\Omega_T)$, $\varphi\in C^2([0,T];H^{2s}(\R^n))$ and $(u_0,u_1)\in H^s(\R^n)\times L^2(\R^n)$ satisfying the compatibility conditions $u_0-\varphi(T)\in \widetilde{H}^s(\Omega)$ and $u_1-\partial_t\varphi(T)\in \widetilde{L}^2(\Omega)$, there exists a unique solution $u^\star$ of
         \begin{equation}
        \label{eq: backwards damped, nonlocal wave equation}
            \begin{cases}
         (L_{-\gamma}+q)v = F & \text{ in } \Omega_T,\\
        v =\varphi & \text{ on } (\Omega_e)_T,\\
          v(T) = u_0,\, \partial_{t}v(T) = u_1 & \text{ on } \Omega.
    \end{cases}
        \end{equation}
   \end{lemma}
   \begin{proof}
       First, note that by the proof of Proposition \ref{prop: Weak solutions to inhomogeneous DNWEQ}, we can assume without loss of generality that $\varphi=0$. Secondly, one easily sees that $\partial_t u^\star=-(\partial_t u)^\star$ and thus a change of variables in \eqref{eq: weak formulation damped nonlocal wave eq} gives the asserted equivalence. The unique solvability of \eqref{eq: backwards damped, nonlocal wave equation} follows from the equivalence and Theorem \ref{thm: weak sol to hom DNWEQ}.
   \end{proof}

    \subsection{Very weak solutions}
    
    Let us start by making some simple observations. Suppose that $u$ and $v$ are smooth solutions of the problems
    \begin{equation}
    \label{eq: motivation very weak solutions 1}
        \begin{cases}
         (L_\gamma+q) u = F & \text{ in } \Omega_T,\\
        u =0 & \text{ on } (\Omega_e)_T,\\
          u(0) = u_0,\, \partial_{t}u(0) = u_1 & \text{ on } \Omega,
    \end{cases}
    \end{equation}
    and 
    \begin{equation}
    \label{eq: motivation very weak solutions 2}
        \begin{cases}
         (L_{-\gamma}+q) v = G & \text{ in } \Omega_T,\\
        v =0 & \text{ on } (\Omega_e)_T,\\
          v(T) = 0,\, \partial_{t}v(T) = 0 & \text{ on } \Omega,
    \end{cases}
    \end{equation}
    respectively. If we multiply the PDE \eqref{eq: motivation very weak solutions 1} by $v$ and integrate over $\Omega_T$, then we get
    \begin{equation}
    \label{eq: formal computation}
    \begin{split}
        &\int_{\Omega_T}Fv\,dxdt=\int_{\Omega_T}[(L_{\gamma}+q)u]v\,dxdt\\
        &= \int_{\Omega}u_0\partial_t v(0)\,dx-\int_{\Omega}u_1v(0)\,dx-\int_\Omega \gamma u_0v(0)\,dx+\int_{\Omega_T}u(L_{-\gamma}+q)v\,dxdt\\
        &= \int_{\Omega}u_0\partial_t v(0)\,dx-\int_{\Omega}u_1v(0)\,dx-\int_\Omega \gamma u_0v(0)\,dx+\int_{\Omega_T}Gu\,dxdt.
    \end{split}
    \end{equation}
    Notice that if $G\in L^2(0,T;\widetilde{L}^2(\Omega))$, $v\in L^2(0,T;\widetilde{H}^s(\Omega))$ and $(v(0),\partial_t v(0))\in \widetilde{H}^s(\Omega)\times\widetilde{L}^2(\Omega)$, then one can make sense of the first integral and the last line in \eqref{eq: formal computation}, even in the case $F\in L^2(0,T;H^{-s}(\Omega))$, $(u_0,u_1)\in \widetilde{L}^2(\Omega)\times H^{-s}(\Omega)$ and $u\in L^2(0,T;\widetilde{L}^2(\Omega))$. Here, $H^{-s}(\Omega)\subset\distr(\Omega)$ is defined by
    \[
        H^{-s}(\Omega)=\{u|_\Omega\,;\,u\in H^{-s}(\R^n)\}
    \]
    and it can be identified with the dual space of $\widetilde{H}^s(\Omega)$, when $\Omega$ is Lipschitz. The previous computation motivates the following definition.

    \begin{definition}[Very weak solutions]
    \label{def: very weak solutions}
        Let $\Omega \subset\R^n$ be a bounded Lipschitz domain, $T>0$, $0<s< 1$ and suppose that $1\leq p\leq \infty$ satisfies \eqref{eq: restrictions on p}. Assume that we have given coefficients $(\gamma,q)\in L^{\infty}(\Omega)\times L^{p}(\Omega)$, source $F\in L^2(0,T;H^{-s}(\Omega))$ and initial conditions $(u_0,u_1)\in \widetilde{L}^2(\Omega)\times H^{-s}(\Omega)$. Then we say that $u\in C([0,T];\widetilde{L}^2(\Omega))\cap C^1([0,T];H^{-s}(\Omega))$ is a \emph{very weak solution} of
         \begin{equation}
        \label{eq: def very weak}
        \begin{cases}
         (L_\gamma+q) u = F & \text{ in } \Omega_T,\\
        u =0 & \text{ on } (\Omega_e)_T,\\
          u(0) = u_0,\, \partial_{t}u(0) = u_1 & \text{ on } \Omega,
    \end{cases}
    \end{equation}
    whenever there holds\footnote{Here and below we sometimes write $\langle \cdot,\cdot\rangle$ to denote the duality pairing between $H^{-s}(\Omega)\times \widetilde{H}^s(\Omega)$.}
    \begin{equation}
    \label{eq: weak formulation of very weak sols}
        \int_0^T \langle G,u\rangle_{L^2(\Omega)}\,dt=\int_0^T\langle F,v\rangle\,dt+\langle u_1,v(0)\rangle-\langle u_0,\partial_t v(0)\rangle_{L^2(\Omega)}+\langle \gamma u_0,v(0)\rangle
    \end{equation}
     for all $G\in L^2(0,T;\widetilde{L}^2(\Omega))$, where $v\in C([0,T];\widetilde{H}^s(\Omega))\cap C^1([0,T];\widetilde{L}^2(\Omega))$ is the unique weak solution of the adjoint equation
      \begin{equation}
    \label{eq: adjoint eq of def very weak}
        \begin{cases}
         (L_{-\gamma}+q) v = G & \text{ in } \Omega_T,\\
        v =0 & \text{ on } (\Omega_e)_T,\\
          v(T) = 0,\, \partial_{t}v(T) = 0 & \text{ on } \Omega,
    \end{cases}
    \end{equation}
    (see Theorem \ref{thm: weak sol to hom DNWEQ}).
    \end{definition}

    Next, let us recall the following well-posedness result of very weak solutions.

    \begin{theorem}[{Very weak solutions for $\gamma=q=0$,\cite[Theorem 3.6]{Optimal-Runge-nonlocal-wave}}]
    \label{thm: well-posedness very weak without damping}
        Let $\Omega \subset\R^n$ be a bounded Lipschitz domain, $T>0$ and $0<s< 1$. Then for any given source $F\in L^2(0,T;H^{-s}(\Omega))$ and initial conditions $(u_0,u_1)\in \widetilde{L}^2(\Omega)\times H^{-s}(\Omega)$, there exists a unique solution to 
          \begin{equation}
        \label{eq: very weak without damping and potential}
        \begin{cases}
         (\partial_t^2+(-\Delta)^s) u = F & \text{ in } \Omega_T,\\
        u =0 & \text{ on } (\Omega_e)_T,\\
          u(0) = u_0,\, \partial_{t}u(0) = u_1 & \text{ on } \Omega
    \end{cases}
    \end{equation}
     and it satisfies the following energy estimate
    \begin{equation}
    \label{eq: energy estimate very weak without damping and potential}
        \|u(t)\|_{L^2(\Omega)}+\|\partial_t u(t)\|_{H^{-s}(\Omega)}\leq C(\|u_0\|_{L^2(\Omega)}+\|u_1\|_{H^{-s}(\Omega)}+\|F\|_{L^2(0,t;H^{-s}(\Omega))})
    \end{equation}
     for all $0\leq t\leq T$.
    \end{theorem}

    Hence, we have a well-defined solution map.

    \begin{proposition}[Solution map]
    \label{prop: solution map}
         Let $\Omega \subset\R^n$ be a bounded Lipschitz domain, $T>0$, $0<s< 1$ and let $X_s\vcentcolon =\widetilde{L}^2(\Omega)\times H^{-s}(\Omega)$ be endowed with the usual product norm  
         \[
            \|(u,w)\|_{X_s}\vcentcolon = (\|u\|_{L^2(\Omega)}^2+\|w\|_{H^{-s}(\Omega)}^2)^{1/2}.
        \] 
        Then the \emph{solution map} $S\colon L^2(0,T;H^{-s}(\Omega))\to C([0,T];X_s)$ defined by
         \begin{equation}
         \label{eq: solution map}
             S(F)\vcentcolon = (u,\partial_t u),
         \end{equation}
         where $u\in C([0,T];\widetilde{L}^2(\Omega))\cap C^1([0,T];H^{-s}(\Omega))$ is the unique solution of \eqref{eq: very weak without damping and potential} with $(u_0,u_1)=0$. Moreover, the solution map is continuous and satisfies the estimate
         \begin{equation}
         \label{eq: continuity estimate solution map}
             \|S(F)(t)\|_{X_s}\leq C\|F\|_{L^2(0,t;H^{-s}(\Omega))}
         \end{equation}
         for any $0\leq t\leq T$.
    \end{proposition}

    \begin{proof}
        First of all note that the solution map $S$ is well-defined by Theorem \ref{thm: well-posedness very weak without damping}. The estimate \eqref{eq: continuity estimate solution map} follows from \eqref{eq: energy estimate very weak without damping and potential}, which together with the linearity of $S$ gives the continuity of $S$. Observe that the linearity of $S$ is a direct consequence of the unique solvability of \eqref{eq: very weak without damping and potential} and the fact that the PDE is linear.
    \end{proof}

    \begin{theorem}
    \label{thm: general well-posedness result of very weak solutions}
        Let $\Omega \subset\R^n$ be a bounded Lipschitz domain, $T>0$, $0<s< 1$ and suppose that $\mathcal{F}\colon C([0,T];X_s)\to L^2(0,T;H^{-s}(\Omega))$ satisfies the Lipschitz estimate
        \begin{equation}
        \label{eq: Lipschitz estimate nonlinearity}
            \|\mathcal{F}(U)(t)-\mathcal{F}(V)(t)\|_{H^{-s}(\Omega)}\leq C\|U(t)-V(t)\|_{X_s}
        \end{equation}
        for a.e.~$0\leq t\leq T$ and $U,V\in C([0,T];X_s)$. Then for all $(u_0,u_1)\in X_s$, there exists a unique solution $u$ of 
        \begin{equation}
        \label{eq: well-posedness nonlinear very weak}
             \begin{cases}
         (\partial_t^2+(-\Delta)^s) u = \mathcal{F}(u,\partial_t u) & \text{ in } \Omega_T,\\
        u =0 & \text{ on } (\Omega_e)_T,\\
          u(0) = u_0,\, \partial_{t}u(0) = u_1 & \text{ on } \Omega,
    \end{cases}
        \end{equation}
        that is the formula \eqref{eq: weak formulation of very weak sols} holds with $F$ replaced by $\mathcal{F}(u,\partial_t u)$ in which we test against every weak solution $v$ of the adjoint equation
        \begin{equation}
        \label{eq: adjoint eq nonlinear}
            \begin{cases}
         (\partial_t^2+(-\Delta)^s) v = G & \text{ in } \Omega_T,\\
        v =0 & \text{ on } (\Omega_e)_T,\\
          v(T) = 0,\, \partial_{t}v(T) = 0 & \text{ on } \Omega
    \end{cases}
        \end{equation}
        with $G\in L^2(0,T;\widetilde{L}^2(\Omega))$.
    \end{theorem}
  
    \begin{proof}[Proof of Theorem \ref{thm: general well-posedness result of very weak solutions}]
        Let $u_h\in C([0,T];\widetilde{L}^2(\Omega))\cap C^1([0,T];H^{-s}(\Omega))$ be the unique solution to 
        \begin{equation}
        \label{eq: homogeneous part}
               \begin{cases}
         (\partial_t^2+(-\Delta)^s) u = 0 & \text{ in } \Omega_T,\\
        u =0 & \text{ on } (\Omega_e)_T,\\
          u(0) = u_0,\, \partial_{t}u(0) = u_1 & \text{ on } \Omega
    \end{cases}
        \end{equation}
        and let us set $U_h\vcentcolon =(u_h,\partial_t u_h)\in C([0,T];X_s)$. Furthermore, we define the operator $\mathcal{T}\colon C([0,T];X_s)\to C([0,T];X_s)$ as
        \begin{equation}
            \mathcal{T}(U)\vcentcolon = U_h+S(\mathcal{F}(U)),
        \end{equation}
        which is well-defined by \eqref{prop: solution map} and the properties of $\mathcal{F}$. Next, we show that $\mathcal{T}$ has a unique fixed point $U=(U_1,U_2)$. 

        \medskip

        \noindent\textit{Step 1.~Existence.} Let $U,V\in C([0,T];X_s)$, then by linearity of $S$, \eqref{eq: continuity estimate solution map} and  \eqref{eq: Lipschitz estimate nonlinearity} we get
        \[
        \begin{split}
        \|\mathcal{T}(U)(t)-\mathcal{T}(V)(t)\|_{X_s}&=\|S(\mathcal{F}(U))(t)-S(\mathcal{F}(V))(t)\|_{X_s}\\
        &=\|S(\mathcal{F}(U)-\mathcal{F}(V))(t)\|_{X_s}\\
        &\leq C\|\mathcal{F}(U)-\mathcal{F}(V)\|_{L^2(0,t;H^{-s}(\Omega))}\\
         &\leq C\|U-V\|_{L^2(0,t;X_s)}.
        \end{split}
        \]
        Next, let us define the following norm on $X_s$
        \begin{equation}
        \label{eq: new norm on Xs}
            \|U\|_{\theta}\vcentcolon = \sup_{0\leq t\leq T}\left(e^{-\theta t}\|U(t)\|_{X_s}\right)
        \end{equation}
        for $\theta>0$, which will be fixed in a moment. Then we have the estimate
        \[
            \|\mathcal{T}(U)(t)-\mathcal{T}(V)(t)\|_{X_s}\leq C\left(\int_0^te^{2\theta \tau}\,d\tau\right)^{1/2}\|U-V\|_{\theta}\leq \frac{C}{(2\theta)^{1/2}}e^{\theta t}\|U-V\|_{\theta}
        \]
        and hence there holds
        \[
         \|\mathcal{T}(U)(t)-\mathcal{T}(V)(t)\|_{\theta}\leq \frac{C}{(2\theta)^{1/2}}\|U-V\|_{\theta}.
        \]
        Therefore, we deduce that $\mathcal{T}$ is a strict contraction from the complete metric space $(C([0,T];X_s),\|\cdot\|_{\theta})$ to itself, when $\theta>0$ is chosen such that $C/(2\theta)^{1/2}<1$. Now, we may invoke Banach's fixed point theorem to obtain a unique fixed point $U=(u,w)$ of $\mathcal{T}$. Next, observe that the definition of the solution map $S$ and $U=\mathcal{T}(U)=U_h+S(\mathcal{F}(U))$ imply
        \[
           u=u_h+u_n \text{ and } w=\partial_t u,
        \]
        where $u_n$ solves 
         \begin{equation}
        \label{eq: nonlinear part}
               \begin{cases}
         (\partial_t^2+(-\Delta)^s) v = \mathcal{F}(U) & \text{ in } \Omega_T,\\
        v =0 & \text{ on } (\Omega_e)_T,\\
          v(0) = 0,\, \partial_{t}v(0) = 0 & \text{ on } \Omega.
    \end{cases}
        \end{equation}
        Going back to the definition of very weak solutions, we see this implies that $u$ solves \eqref{eq: well-posedness nonlinear very weak}. 

        \medskip
        
        \noindent\textit{Step 2.~Uniqueness.} Suppose $\Tilde{u}\in C([0,T];\widetilde{L}^2(\Omega))\cap C^1([0,T];H^{-s}(\Omega))$ is any other solution to \eqref{eq: well-posedness nonlinear very weak}, then $\bar{u}\vcentcolon = u-\widetilde{u}$ solves
        \begin{equation}
         \label{eq: uniquenes very weak nonlinear}
             \begin{cases}
         (\partial_t^2+(-\Delta)^s) v = \mathcal{F}(u,\partial_t u)-\mathcal{F}(\widetilde{u},\partial_t \widetilde{u}) & \text{ in } \Omega_T,\\
        v =0 & \text{ on } (\Omega_e)_T,\\
          v(0) = 0,\, \partial_{t}v(0) = 0 & \text{ on } \Omega.
    \end{cases}
        \end{equation}
        Thus, applying the energy estimate \eqref{eq: energy estimate very weak without damping and potential} together with the Lipschitz assumption on $\mathcal{F}$, we see that
        \[
        \begin{split}
            \|U(t)-\widetilde{U}(t)\|^2_{X_s}&\leq C\int_0^t\|\mathcal{F}(U)(\tau)-\mathcal{F}(\widetilde{U})(\tau)\|_{H^{-s}(\Omega)}^2\,d\tau\\
            &\leq C\int_0^t\|U(t)-\widetilde{U}(t)\|^2_{X_s}\,d\tau,
        \end{split}
        \]
        where $U=(u,\partial_tu)$ and $\widetilde{U}=(\widetilde{u},\partial_t\widetilde{u})$. So, Gronwall's inequality shows that $u=\widetilde{u}$. This establishes the uniqueness assertion and we can conclude the proof.
    \end{proof}

    As an application of Theorem \ref{thm: general well-posedness result of very weak solutions}, we can show the unique solvability of \eqref{eq: damped nonlocal wave equations well-posedness} for rough source and initial data.

    \begin{theorem}[Very weak solutions to DNWEQ]
    \label{thm: very weak sol DNWEQ}
        Let $\Omega \subset\R^n$ be a bounded Lipschitz domain, $T>0$, $0<s<\alpha\leq 1$ and suppose that $1\leq p\leq \infty$ satisfies \eqref{eq: restrictions on p}. Assume that we have given coefficients $(\gamma,q)\in C^{0,\alpha}(\R^n)\times L^{p}(\Omega)$. Then for any $F\in L^2(0,T;H^{-s}(\Omega))$ and $(u_0,u_1)\in \widetilde{L}^2(\Omega)\times H^{-s}(\Omega)$, there exists a unique solution of
        \begin{equation}
             \label{eq: well-posedness very weak DNWEQ}
        \begin{cases}
         (L_\gamma+q) u = F & \text{ in } \Omega_T,\\
        u =0 & \text{ on } (\Omega_e)_T,\\
          u(0) = u_0,\, \partial_{t}u(0) = u_1 & \text{ on } \Omega.
    \end{cases}
        \end{equation}
    \end{theorem}

    \begin{proof}
        Let us define the mapping $\mathcal{F}\colon C([0,T];X_s)\to L^2(0,T;H^{-s}(\Omega))$ by
        \[
            \mathcal{F}(U)(t)\vcentcolon =F-\gamma w(t)-qu(t),
        \]
        where $U=(u,w)\in C([0,T];X_s)$. On the one hand, using the estimate \eqref{eq: L2 estimate potential} we see that for any $u\in \widetilde{L}^2(\Omega)$ one has $qu\in H^{-s}(\Omega)$ and there holds
        \begin{equation}
        \label{eq: continuity estimate dual potential} 
            \begin{split}
                \|qu\|_{H^{-s}(\Omega)}
                &=\sup_{\|v\|_{\widetilde{H}^s(\Omega)}\leq 1}|\langle u,qv\rangle_{L^2(\Omega)}|\\
                &\leq C\|q\|_{L^p(\Omega)}\|u\|_{L^2(\Omega)}.
            \end{split}
        \end{equation}
        On the other hand, by applying \cite[Lemma 3.1]{Stability-fractional-conductivity} and $\partial\Omega\in C^0$ we deduce that for any $v\in \widetilde{H}^s(\Omega)$ one has $\gamma v\in \widetilde{H}^s(\Omega)$ and it obeys the estimate
        \begin{equation}
        \label{eq: multiplication by gamma}
            \|\gamma v\|_{\widetilde{H}^s(\Omega)}\leq C\|\gamma\|_{C^{0,\alpha}(\R^n)}\|v\|_{\widetilde{H}^s(\Omega)}.
        \end{equation}
        Thus, we can again infer from a duality argument that $H^{-s}(\Omega)\ni w\mapsto \gamma w\in H^{-s}(\Omega)$ is a continuous map satisfying
        \begin{equation}
        \label{eq: continuity estimate dual damping} 
             \begin{split}
                \|\gamma w\|_{H^{-s}(\Omega)}&=\sup_{\|v\|_{\widetilde{H}^s(\Omega)}\leq 1}|\langle \gamma w,v\rangle|\\
                &=\sup_{\|v\|_{\widetilde{H}^s(\Omega)}\leq 1}|\langle w,\gamma v\rangle|\\
                &\leq C\|\gamma\|_{C^{0,\alpha}(\R^n)}\|w\|_{H^{-s}(\Omega)}.
            \end{split}
        \end{equation}
        From the estimates \eqref{eq: continuity estimate dual potential} and \eqref{eq: continuity estimate dual damping}, we easily deduce that $\mathcal{F}$ is well-defined and satisfies the Lipschitz estimate 
        \begin{equation}
        \label{eq: Lipschitz estimate application}
            \|\mathcal{F}(U)(t)-\mathcal{F}(V)(t)\|_{H^{-s}(\Omega)}\leq C(\|\gamma\|_{C^{0,\alpha}(\R^n)}+\|q\|_{L^p(\Omega)})\|U(t)-V(t)\|_{X_s}
        \end{equation}
        for all $U,V\in C([0,T];X_s)$. Thus, we can apply Theorem  \ref{thm: general well-posedness result of very weak solutions} to get the existence of a unique solution to \eqref{eq: well-posedness very weak DNWEQ} in the sense that for any $G\in L^2(0,T;\widetilde{L}^2(\Omega))$ and corresponding solution $v$ of \eqref{eq: adjoint eq nonlinear}, there holds
        \begin{equation}
        \label{eq: prelim def of very weak sol}
         \int_0^T \langle G,u\rangle_{L^2(\Omega)}\,dt=\int_0^T\langle (F-\gamma \partial_t u-qu),v\rangle\,dt+\langle u_1,v(0)\rangle-\langle u_0,\partial_t v(0)\rangle_{L^2(\Omega)}.
        \end{equation}
        It remains to verify that $u$ is indeed a solution of \eqref{eq: well-posedness very weak DNWEQ} in the sense of Definition \ref{def: very weak solutions}. For this purpose let $G\in L^2(0,T;\widetilde{L}^2(\Omega))$ and suppose that $v$ is the unique solution to \eqref{eq: adjoint eq of def very weak}. Hence, $v$ solves
        \[
        \begin{cases}
         (\partial_t^2+(-\Delta)^s) v = \widetilde{G} & \text{ in } \Omega_T,\\
        v =0 & \text{ on } (\Omega_e)_T,\\
          v(T) = 0,\, \partial_{t}v(T) = 0 & \text{ on } \Omega
    \end{cases}
        \]
        with $\widetilde{G}=G+\gamma \partial_t v-qv\in L^2(0,T;\widetilde{L}^2(\Omega))$ (see \eqref{eq: L2 estimate potential}). Next, we claim that there holds
        \begin{equation}
        \label{eq: integration by parts formula}
            \int_0^T \langle \gamma \partial_t u,v\rangle\,dt=-\int_0^T \langle \gamma\partial_t v,u\rangle_{L^2(\Omega)}\,dt-\langle \gamma u_0,v(0)\rangle.
        \end{equation}
        For this purpose, let us consider for $\eps>0$ the unique solution $v_\eps\in H^1(0,T;\widetilde{H}^s(\Omega))$ with $\partial_t^2 v_\eps\in L^2(0,T;H^{-s}(\Omega))$ to the following parabolically regularized problem
        \begin{equation}
				\label{eq: regularization for equation of w}
				\begin{cases}
					(\partial_t^2 -\eps (-\Delta)^s \partial_t +(-\Delta)^s)v = \widetilde{G}&\text{ in }\Omega_T,\\
					v=0  &\text{ in }(\Omega_e)_T,\\
					v(T)= \partial_t v(T)=0 &\text{ in }\Omega
				\end{cases}
		\end{equation}
        (see \cite[Chapter XVIII, Section 5.3.1]{DautrayLionsVol5}). By \cite[Chapter XVIII, Section 5.3.4]{DautrayLionsVol5} we know that there holds
        \begin{equation}
        \label{eq: convergence}
            \begin{split}
					v_{\eps}&\weakstar v \text{ in }L^{\infty}(0,T;\widetilde{H}^s(\Omega)),\\
					\partial_t v_{\eps}&\weakstar \partial_t v \text{ in }L^{\infty}(0,T;\widetilde{L}^2(\Omega)),\\
					v_\eps(t)&\to v(t) \text{ in } \widetilde{H}^s(\Omega)\text{ for all }0\leq t\leq T.
				\end{split}
        \end{equation}
        First, note that the conditions $u\in C^1([0,T];H^{-s}(\Omega))$ and $v_\eps\in C^1([0,T];\widetilde{L}^2(\Omega))$, where the latter follows from the Sobolev embedding, guarantee that $\langle \gamma u,v_\eps\rangle\in C^1([0,T])$ with
        \begin{equation}
        \label{eq: product rule}
            \partial_t \langle \gamma u,v_\eps\rangle=\langle \partial_t u,\gamma v_\eps\rangle+\langle u,\gamma \partial_t v_\eps\rangle_{L^2(\Omega)}.
        \end{equation}
        Thus, by the fundamental theorem of calculus we deduce that there holds
        \[
            \langle \gamma u(T),v_\eps(T)\rangle-\langle \gamma u_0,v_\eps(0)\rangle=\int_0^T\langle \partial_t u,\gamma v_\eps\rangle+\langle u,\gamma \partial_t v_\eps\rangle_{L^2(\Omega)}\,dt.
        \]
        By the convergence assertions \eqref{eq: convergence} and $v_\eps (T)=0$, we get
        \[
            -\langle \gamma u_0,v(0)\rangle=\int_0^T\langle \partial_t u,\gamma v\rangle+\langle u,\gamma \partial_t v\rangle_{L^2(\Omega)}\,dt.
        \]
        This proves \eqref{eq: integration by parts formula}. Hence, inserting this into \eqref{eq: prelim def of very weak sol}, we obtain
        \[
        \begin{split}
            \int_0^T\langle\widetilde{G},u\rangle_{L^2(\Omega)}\,dt&=\int_0^T\langle (F-\gamma \partial_t u-qu),v\rangle\,dt+\langle u_1,v(0)\rangle-\langle u_0,\partial_t v(0)\rangle_{L^2(\Omega)}\\
            &=\int_0^T\langle F,v\rangle\,dt+\int_0^T \langle u,\gamma \partial_t v\rangle_{L^2(\Omega)}\,dt-\int_0^T\langle u,qv\rangle\,dt\\
            &\quad+\langle u_1,v(0)\rangle-\langle u_0,\partial_t v(0)\rangle_{L^2(\Omega)}+\langle \gamma u_0,v(
        0)\rangle.
        \end{split}
        \]
        As $\widetilde{G}=G+\gamma \partial_t v-qv$ this gives
        \[
        \begin{split}
            \int_0^T \langle G,u\rangle_{L^2(\Omega)}\,dt&=\int_0^T \langle F,v\rangle\, dt+\langle u_1,v(0)\rangle-\langle u_0,\partial_t v(0)\rangle_{L^2(\Omega)}+\langle \gamma u_0,v(0)\rangle.
        \end{split}
        \]
        Hence, we observe that $u$ is indeed a solution of \eqref{eq: well-posedness very weak DNWEQ} in the sense of Definition \ref{def: very weak solutions}. By reversing the above arguments one can also observe that if $u$ is a solution in the sense of Definition \ref{def: very weak solutions}, then by \eqref{eq: integration by parts formula} it is a solution in the sense of \eqref{eq: prelim def of very weak sol} and thus the solution in the sense of Definition \ref{def: very weak solutions} is unique.
    \end{proof}

    \section{The inverse problem}
    \label{sec: inverse problem}

    After establishing the theory of very weak solutions to damped, nonlocal wave equations, we now turn our attention to the inverse problem part. First, in Section \ref{sec: Runge approx} we prove the optimal Runge approximation theorem (Theorem \ref{thm: Runge approx}) and in Section \ref{sec: integral identity} a suitable integral identity. Using these results, we then show in Section  \ref{sec: linear uniqueness} our first main result dealing with linear perturbations (Theorem \ref{thm: uniqueness linear}). Finally, in Section \ref{sec: semilinear uniqueness} we prove Theorem \ref{thm: uniqueness semilinear} showing that the damping coefficient and the nonlinearity can be determined simultaneously.

    \subsection{Runge approximation}
    \label{sec: Runge approx}

    With the material from Section \ref{sec: Very weak solutions to damped, nonlocal wave equations} at our disposal, we can now show the following Runge approximation theorem, whose proof is very similar to the one of \cite[Theorem 1.2]{Optimal-Runge-nonlocal-wave}.

    \begin{theorem}[Runge approximation]
    \label{thm: Runge approx}
        Let $\Omega \subset\R^n$ be a bounded Lipschitz domain, $T>0$, $0<s<\alpha\leq 1$ and suppose that $1\leq p\leq \infty$ satisfies \eqref{eq: restrictions on p}. Assume that we have given coefficients $(\gamma,q)\in C^{0,\alpha}(\R^n)\times L^{p}(\Omega)$. Then for any measurement set $W\subset\Omega_e$ and initial conditions $(u_0,u_1)\in \widetilde{H}^s(\Omega)\times\widetilde{L}^2(\Omega)$, the \emph{Runge set}
        \begin{equation}
        \label{eq: Runge set}
            \mathcal{R}^{u_0,u_1}_W\vcentcolon = \{u_\varphi-\varphi\,;\,\varphi\in C_c^{\infty}(W_T)\}
        \end{equation}
        is dense in $L^2(0,T;\widetilde{H}^s(\Omega))$, where $u_\varphi$ is the unique solution to
        \begin{equation}
        \label{eq: PDE Runge}
            \begin{cases}
         (L_{\gamma}+q)u = 0 & \text{ in } \Omega_T,\\
        u =\varphi & \text{ on } (\Omega_e)_T,\\
          u(0) = u_0,\, \partial_{t}u(0) = u_1 & \text{ on } \Omega,
    \end{cases}
        \end{equation}
        (see Proposition \ref{prop: Weak solutions to inhomogeneous DNWEQ}).
    \end{theorem}

    \begin{proof}
        First of all note that it is enough to consider the case $(u_0,u_1)=0$. To see this assume that the density holds for $\mathcal{R}_W\vcentcolon =\mathcal{R}_W^{0,0}$ and let $f\in L^2(0,T;\widetilde{H}^s(\Omega))$. Let $v_0$ be the unique solution to
        \begin{equation}
              \begin{cases}
         (L_{\gamma}+q)v = 0 & \text{ in } \Omega_T,\\
        v =0 & \text{ on } (\Omega_e)_T,\\
          v(0) = u_0,\, \partial_{t}v(0) = u_1 & \text{ on } \Omega
    \end{cases}
        \end{equation} 
        and define $\widetilde{f}\vcentcolon = f-v_0\in L^2(0,T;\widetilde{H}^s(\Omega))$. By assumption there exists $(\varphi_k)_{k\in\N}\subset C_c^{\infty}(W_T)$ such that $u_k-\varphi_k\to \widetilde{f}$ in $L^2(0,T;\widetilde{H}^s(\Omega))$ as $k\to \infty$, where $u_k$ is the unique solution to 
        \begin{equation}
        \label{eq: PDE Runge vanishing initial}
              \begin{cases}
         (L_{\gamma}+q)u = 0 & \text{ in } \Omega_T,\\
        u =\varphi & \text{ on } (\Omega_e)_T,\\
          u(0) = 0,\, \partial_{t}u(0) = 0 & \text{ on } \Omega
    \end{cases}
        \end{equation} 
        with $\varphi=\varphi_k$. Then $v_k\vcentcolon = u_k+v_0$ is the unique solution to \eqref{eq: PDE Runge} with $\varphi=\varphi_k$. The above convergence now implies $v_k-\varphi_k\to f$ in $L^2(0,T;\widetilde{H}^s(\Omega))$ as $k\to\infty$ and we get that $\mathcal{R}_W^{u_0,u_1}$ is dense in $L^2(0,T;\widetilde{H}^s(\Omega))$.

        Therefore, it remains to show that $\mathcal{R}_W$ is dense in $L^2(0,T;\widetilde{H}^s(\Omega))$. As usual, we show this by a Hahn--Banach argument. Thus, suppose that $F\in L^2(0,T;H^{-s}(\Omega))$ vanishes on $\mathcal{R}_W$. Let us recall that if $\varphi\in C_c^{\infty}(W_T)$ and $u$ solves \eqref{eq: PDE Runge vanishing initial}, then by \eqref{eq: Usual cauchy homogeneous} and Lemma \ref{lemma: time reversal of solution} the function $v=(u-\varphi)^\star$ satisfies
		\begin{equation}
			\begin{cases}
				(L_{-\gamma}+q)v = -(-\Delta)^s\varphi^\star &\text{ in }\Omega_T,\\
				v=0  &\text{ in }(\Omega_e)_T,\\
				v(T)=\partial_t v(T)=0 &\text{ in }\Omega.
			\end{cases}
		\end{equation}
        Next, let $w$ be the unique solution to
		\begin{equation}
			\label{eq: solution to weak RHS}
			\begin{cases}
				(L_{\gamma}+q)w= F^\star &\text{ in }\Omega_T,\\
				w=0  &\text{ in }(\Omega_e)_T,\\
				w(0)= \partial_t w(0)=0 &\text{ in }\Omega
			\end{cases}
		\end{equation}
        (see Theorem \ref{thm: very weak sol DNWEQ}). By testing the equation for $w$ by $v$, we get
		\[
		-\int_0^T \langle (-\Delta)^s\varphi^\star, w\rangle_{L^2(\Omega)}\,dt=\int_0^T\langle F^\star,v\rangle\,dt=\int_0^T\langle F,u_\varphi-\varphi\rangle\,dt=0
		\]
        for any $\varphi\in C_c^{\infty}(W_T)$. This ensures that there holds
		\[
		(-\Delta)^s w=0\quad \text{ in }W_T.
		\]
        Furthermore, by construction $w$ vanishes in $(\Omega_e)_T$ and hence the unique continuation principle for the fractional Laplacian guarantees $w=0$ in $\R^n_T$ (see \cite{GSU20}). As very weak solutions are distributional solutions, we get
		\[
		\int_0^T\langle F^\star,\Phi\rangle\,dt=\int_0^T \langle (L_{-\gamma}+q)\Phi,w\rangle_{L^2(\Omega)}\,dt=0
		\]
		for all $\Phi\in C_c^{\infty}(\Omega_T)$. To see that very weak solutions are distributional solutions, one can simply take $G=\chi_\Omega(L_\gamma+q)\Phi$ with $\Phi\in C_c^{\infty}(\Omega\times [0,T))$ in Definition \ref{def: very weak solutions}, where $\chi_\Omega$ denotes the characteristic function of $\Omega$ (see also \cite[Proposition 3.8]{Optimal-Runge-nonlocal-wave}). By density of  $C_c^{\infty}(\Omega_T)$ in $L^2(0,T;\widetilde{H}^s(\Omega))$ we deduce that $F=0$.
        This concludes the proof.
    \end{proof}

    As a consequence we have the following lemma.

    \begin{lemma}[Convergence of time derivative]
    \label{lemma: convergence of time derivatives}
    Let $\Omega \subset\R^n$ be a bounded Lipschitz domain, $T>0$, $0<s<\alpha\leq 1$ and suppose that $1\leq p\leq \infty$ satisfies \eqref{eq: restrictions on p}. Assume that we have given coefficients $(\gamma,q)\in C^{0,\alpha}(\R^n)\times L^{p}(\Omega)$. Let $\Phi,\Psi\in L^2(0,T;\widetilde{H}^s(\Omega))\cap H^1(0,T;H^{-s}(\Omega))$ and suppose $(\varphi_k)_{k\in\N}\subset C_c^{\infty}((\Omega_e)_T)$ is such that
    \begin{equation}
    \label{eq: convergence assertion for time derivative}
        u_k-\varphi_k\to \Phi\text{ in }L^2(0,T;\widetilde{H}^s(\Omega))\text{ as }k\to\infty,
    \end{equation}
    where $u_k$ solves 
    \begin{equation}
         \label{eq: approx time derivative}
            \begin{cases}
         (L_{\gamma}+q)u = 0 & \text{ in } \Omega_T,\\
        u =\varphi_k & \text{ on } (\Omega_e)_T,\\
          u(0) = 0,\, \partial_{t}u(0) = 0 & \text{ on } \Omega
    \end{cases}
    \end{equation}
    for $k\in\N$. If $\Phi,\Psi$ satisfy one of the conditions
        \begin{enumerate}[(a)]
            \item \label{item cond 1 first} $\Psi(T)=\Phi(0)=0$
            \item \label{item cond 2 first} or $\Psi(T)=\Psi(0)=0$,
        \end{enumerate}
        then we have
        \begin{equation}
        \label{eq: first order limit}
            \lim_{k\to\infty}\int_0^T\langle \partial_t (u_k-\varphi_k),\Psi\rangle \,dt =\int_0^T\langle \partial_t\Phi,\Psi\rangle\,dt.
        \end{equation}
    \end{lemma}
    \begin{remark}
        Let us note that the same formula \eqref{eq: first order limit} holds for second order time derivatives under appropriate conditions.
    \end{remark}

    \begin{proof}
      Using the integration by parts formula, we may compute
        \[\small
        \begin{split}
             &\lim_{k\to\infty}\int_0^T\langle \partial_t (u_k-\varphi_k),\Psi\rangle \,dt \\
             &= \lim_{k\to\infty}\left(\langle (u_k-\varphi_k)(T),\Psi(T)\rangle_{L^2(\Omega)}-\langle (u_k-\varphi_k)(0),\Psi(0)\rangle_{L^2(\Omega)}-\int_0^T\langle  \partial_t\Psi,u_k-\varphi_k\rangle \,dt\right)\\
             &=-\lim_{k\to\infty}\int_0^T\langle  \partial_t\Psi,u_k-\varphi_k\rangle \,dt\\
             &=-\int_0^T \langle \partial_t \Psi,\Phi\rangle \, dt\\
             &=\langle \Phi(0),\Psi(0)\rangle_{L^2(\Omega)}-\langle \Phi(T),\Psi(T)\rangle_{L^2(\Omega)}+\int_0^T \langle \partial_t\Phi, \Psi\rangle \, dt\\
             &=\int_0^T \langle \partial_t\Phi, \Psi\rangle \, dt.
        \end{split}
        \]
        In the first equality sign we used an integration by parts, in the second equality we used \eqref{eq: approx time derivative}, $\Psi(T)=0$ and \eqref{eq: approx time derivative}, in the third equality the convergence \eqref{eq: convergence assertion for time derivative}, in the fourth equality again an integration by parts and finally in the last equality the conditions \ref{item cond 1 first} or \ref{item cond 2 first}. 
    \end{proof}

    \subsection{DN map and integral identities}
    \label{sec: integral identity}    

    Next, we define the \emph{Dirichlet to Neumann (DN) map} $\Lambda_{\gamma,q}$ related to 
    \begin{equation}
    \label{eq: PDE for integral identity}
         \begin{cases}
         (L_{\gamma}+q)u = 0 & \text{ in } \Omega_T,\\
        u =\varphi & \text{ on } (\Omega_e)_T,\\
          u(0) = 0,\, \partial_{t}u(0) = 0 & \text{ on } \Omega,
    \end{cases}
    \end{equation}
    via
    \begin{equation}
    \label{eq: DN map}
        \langle \Lambda_{\gamma,q}\varphi,\psi\rangle=\int_{\R^n_T}(-\Delta)^{s/2} u_\varphi (-\Delta)^{s/2}\psi\,dx
    \end{equation}
    for all $\varphi,\psi\in C_c^{\infty}((\Omega_e)_T)$, where $u_\varphi$ is the unique solution to \eqref{eq: PDE for integral identity} with exterior condition $\varphi$. Using the above preparation, we now establish the following integral identity.
    
    \begin{proposition}[Integral identity for linear perturbations]
    \label{prop: integral identity}
        Let $\Omega \subset\R^n$ be a bounded Lipschitz domain, $T>0$, $0<s<\alpha\leq 1$ and suppose that $1\leq p\leq \infty$ satisfies \eqref{eq: restrictions on p}. Assume that we have given coefficients $(\gamma_j,q_j)\in C^{0,\alpha}(\R^n)\times L^{p}(\Omega)$ for $j=1,2$. Let $\varphi_j\in C_c^{\infty}((\Omega_e)_T)$ and denote by $u_j$ the corresponding solution of \eqref{eq: PDE for integral identity} with $(\gamma,q)=(\gamma_j,q_j)$. Then there holds
       \begin{equation}
       \label{eq: integral identity}
       \begin{split}
           &\langle (\Lambda_{\gamma_1,q_1}-\Lambda_{\gamma_2,q_2})\varphi_1,\varphi_2^\star\rangle\\
           &\quad =\int_{\Omega_T}\{[(\gamma_1-\gamma_2)\partial_t+q_1-q_2](u_1-\varphi_1)\}(u_2-\varphi_2)^\star\,dxdt.
       \end{split}
       \end{equation}
    \end{proposition}

    \begin{proof}
        Let $(\Gamma_j,Q_j)\in C^{0,\alpha}(\R^n)\times L^p(\Omega)$, $j=1,2$, and suppose $U_j$ is the unique solutions of \eqref{eq: PDE for integral identity} with $(\gamma,q)=(\Gamma_j,Q_j)$ and exterior condition $\varphi=\psi_j$. Then we may compute
        \begin{equation}
        \label{eq: calculation for integral identity}
            \begin{split}
                &\int_{\Omega_T}\{[(\Gamma_1-\Gamma_2)\partial_t+Q_1-Q_2](U_1-\psi_1)\}(U_2-\psi_2)^\star\,dxdt\\
                & =\int_{\Omega_T}\{[\Gamma_1\partial_t+Q_1](U_1-\psi_1)\}(U_2-\psi_2)^\star\,dxdt\\
                &\quad -\int_{\Omega_T}(U_1-\psi_1)[- \Gamma_2\partial_t+Q_2](U_2-\psi_2)^\star\,dxdt\\
                &=\int_{0}^T\langle L_{\Gamma_1,Q_1}(U_1-\psi_1),(U_2-\psi_2)^\star\rangle\,dt\\
                &\quad -\int_{0}^T\langle (\partial_t^2+(-\Delta)^s)(U_1-\psi_1),(U_2-\psi_2)^\star\rangle\,dt\\
                &\quad -\int_{0}^T\langle L_{-\Gamma_2,Q_2}(U_2-\psi_2)^\star,(U_1-\psi_1)\rangle\,dt\\
                &\quad +\int_{0}^T\langle (\partial_t^2+(-\Delta)^s)(U_2-\psi_2)^\star,(U_1-\psi_1)\rangle\,dt\\
                &=-\int_{0}^T\langle (-\Delta)^s\psi_1,(U_2-\psi_2)^\star\rangle_{L^2(\Omega)}\,dt +\int_{0}^T\langle (-\Delta)^s \psi_2^\star,(U_1-\psi_1)\rangle_{L^2(\Omega)}\,dt\\
                &=\int_{\R^n_T} ((-\Delta)^s\psi_2^\star) U_1\,dxdt-\int_{\R^n_T} ((-\Delta)^s\psi_1) U_2^\star\,dxdt\\
                &=\langle \Lambda_{\Gamma_1,Q_1}\psi_1,\psi_2^\star\rangle-\langle \Lambda_{\Gamma_2,Q_2}\psi_2,\psi_1^\star\rangle.
            \end{split}
        \end{equation}
        In the first equality we used that $U_1-\psi_1$ has vanishing initial conditions, $(U_2-\psi_2)^\star$ has vanishing terminal conditions and an integration by parts. In the third equality we used that the PDEs for $U_1-\psi_1$ and $(U_2-\psi_2)^\star$ hold in the sense of $L^2(0,T;H^{-s}(\Omega))=(L^2(0,T;\widetilde{H}^s(\Omega))'$ (see Lemma \ref{lemma: time reversal of solution}). In the fourth equality, we used the PDEs for $U_1$ and $U_2$, Lemma \ref{lemma: time reversal of solution} and that there holds
        \[
        \begin{split}
             &\int_{0}^T\langle (\partial_t^2+(-\Delta)^s)(U_2-\psi_2)^\star,(U_1-\psi_1)\rangle\,dt\\
             &=\int_{0}^T\langle (\partial_t^2+(-\Delta)^s)(U_1-\psi_1),(U_2-\psi_2)^\star\rangle\,dt,
        \end{split}
        \]
        which can be established similarly as \cite[Claim 4.2]{Semilinear-nonlocal-wave-eq} (see also the proof of Theorem \ref{thm: very weak sol DNWEQ}). In the last equality, we have made the change of variables $\tau=T-t$ for the second integral. On the one hand, using \eqref{eq: calculation for integral identity} with
        \[
            \Gamma_1=\Gamma_2=\gamma_j\text{ and }Q_1=Q_2=q_j,
        \]
        we observe that 
        \begin{equation}
        \label{eq: self-adjointness DN map}
            \langle \Lambda_{\gamma_j,q_j}\psi_1,\psi_2^\star\rangle=\langle \Lambda_{\gamma_j,q_j}\psi_2,\psi_1^\star\rangle
        \end{equation}
        for all $\psi_j\in C_c^{\infty}((\Omega_e)_T)$, $j=1,2$. On the other hand, choosing 
        \[
            \Gamma_j=\gamma_j,\,Q_j=q_j\text{ and }\psi_j=\varphi_j
        \]
        in \eqref{eq: calculation for integral identity} and taking into account the self-adjointness  \eqref{eq: self-adjointness DN map}, we get \eqref{eq: integral identity}.
        \end{proof}

    \subsection{Simultaneous determination of damping coefficient and linear perturbations}
    \label{sec: linear uniqueness}

    \begin{proof}[Proof of Theorem \ref{thm: uniqueness linear}]
        First note that by the integral identity in Proposition \ref{prop: integral identity}, we may deduce from the condition \eqref{eq: equality of DN maps} that there holds
        \begin{equation}
        \label{eq: uniqueness proof help identity}
        \int_{\Omega_T}\{[(\gamma_1-\gamma_2)\partial_t+q_1-q_2](u_1-\varphi_1)\}(u_2-\varphi_2)^\star\,dxdt=0
        \end{equation}
        for all $\varphi_j\in C_c^{\infty}((W_j)_T)$, where $u_j$ denotes the unique solution to 
        \begin{equation}
        \label{eq: PDE uniqueness proof}
            \begin{cases}
         (L_{\gamma_j}+q_j)u = 0 & \text{ in } \Omega_T,\\
        u =\varphi_j & \text{ on } (\Omega_e)_T,\\
          u(0) = 0,\, \partial_{t}u(0) = 0 & \text{ on } \Omega.
    \end{cases}
        \end{equation}
        Let $\omega\Subset \Omega$ and choose a cutoff function $\Phi_1\in C_c^{\infty}(\Omega)$ satisfying $\Phi_1=1$ on $\omega$. Moreover, let $\Phi_2\in C_c^{\infty}(\omega_T)$. By the Runge approximation (Theorem \ref{thm: Runge approx}), there exist sequences $(\varphi_j^k)_{k\in\N}\subset C_c^{\infty}((W_j)_T)$ with corresponding solutions $u_j^k$ of \eqref{eq: PDE uniqueness proof} with $\varphi_j=\varphi_j^k$ such that $u_j^k-\varphi_j^k\to \Phi_j$ in $L^2(0,T;\widetilde{H}^s(\Omega))$. 
        Taking $\varphi_1=\varphi_1^k$ and $\varphi_2=\varphi_2^{\ell}$ in \eqref{eq: uniqueness proof help identity} gives
        \[
            \int_{\Omega_T}\{[(\gamma_1-\gamma_2)\partial_t+q_1-q_2](u^k_1-\varphi^k_1)\}(u^{\ell}_2-\varphi^{\ell}_2)^\star\,dxdt=0
        \]
        for all $k,\ell\in\N$. First, we let $\ell\to\infty$ to deduce
        \begin{equation}
         \label{eq: uniqueness proof help identity 2}
            \int_{\Omega_T}\{[(\gamma_1-\gamma_2)\partial_t+q_1-q_2](u^k_1-\varphi^k_1)\}\Phi_2^\star\,dxdt=0
        \end{equation}
        for all $k\in\N$. As $\gamma_1-\gamma_2\in C^{0,\alpha}(\R^n)$ the estimate \eqref{eq: multiplication by gamma} ensures that we can apply Lemma \ref{lemma: convergence of time derivatives} under the condition \ref{item cond 2 first} and so $\partial_t\Phi_1=0$ shows that the first term in \eqref{eq: uniqueness proof help identity 2} goes to zero. So in the limit $k\to\infty$ what remains is
        \[
            \int_{\Omega_T}(q_1-q_2)\Phi_2^\star \,dxdt=0,
        \]
        where we used $\Phi_1=1$ on $\omega$. This ensures that $q_1=q_2$ on $\omega$. As the set $\omega$ is arbitrary, we get $q_1=q_2$ in $\Omega$. Now, the identity \eqref{eq: uniqueness proof help identity} reduces to
        \[
        \int_{\Omega_T}\{[(\gamma_1-\gamma_2)\partial_t](u_1-\varphi_1)\}(u_2-\varphi_2)^\star\,dxdt=0
        \]
        for all $\varphi_j\in C_c^{\infty}((W_j)_T)$. We choose $\eta\in C_c^{\infty}(\Omega_T)$, define
        \[
            \Phi_1(x,t)=\int_0^t\eta(x,\tau)\,d\tau\in C_c^{\infty}(\Omega\times (0,T])
        \]
        and take $\Phi_2\in C_c^{\infty}(\Omega_T)$. Then using $\partial_t \Phi_1=\eta$ and arguing as above via a Runge approximation and Lemma \ref{lemma: convergence of time derivatives}, we get from \eqref{eq: uniqueness proof help identity} the identity
        \[
            \int_{\Omega_T}(\gamma_1-\gamma_2)\eta\Phi_2^\star\,dxdt=0.
        \]
        This again implies $\gamma_1=\gamma_2$ in $\Omega$.
    \end{proof}

    \subsection{Simultaneous determination of damping coefficient and nonlinearity}
    \label{sec: semilinear uniqueness}

    Before turning to the proof of our second main result, let us recall that the \emph{DN map} related to the problem
    \begin{equation}
         \label{eq: PDE for semilinear problem}
         \begin{cases}
         L_{\gamma}u+f(u) = 0 & \text{ in } \Omega_T,\\
        u =\varphi & \text{ on } (\Omega_e)_T,\\
          u(0) = 0,\, \partial_{t}u(0) = 0 & \text{ on } \Omega
    \end{cases}
    \end{equation}
    is defined by
    \begin{equation}
        \langle \Lambda_{\gamma,f}\varphi,\psi\rangle\vcentcolon = \int_{\R^n_T}(-\Delta)^{s/2}u_\varphi(-\Delta)^{s/2}\psi\,dxdt,
    \end{equation}
    where $\varphi,\psi\in C_c^{\infty}((\Omega_e)_T)$ and $u_\varphi$ is the unique solution to \eqref{eq: PDE for semilinear problem} (see \cite[Proposition 3.7]{Semilinear-nonlocal-wave-eq}).

    \begin{proof}[Proof of Theorem \ref{thm: uniqueness semilinear}]
        Let $\eps>0$ and denote by $u^{(j)}_\eps$ the unique solutions to  \eqref{eq: PDE for semilinear problem} with $f=f_j$, $\gamma=\gamma_j$ and $\varphi=\eps \eta$ for some fixed $\eta\in C_c^{\infty}((W_1)_T)$. Let us observe that the UCP for the fractional Laplacian and the condition \eqref{eq: equality of DN maps semilinear} imply that $u_\eps\vcentcolon = u^{(1)}_\eps=u^{(2)}_\eps$. Next, let us note that we can write
        \begin{equation}
        \label{eq: decomposition of u eps}
            u_\eps=\eps v_j+R^{(j)}_\eps
        \end{equation}
        for $j=1,2$, where $v_j$ and $R^{(j)}_\eps$ are the unique solutions of
        \begin{equation}
        \label{eq: linear part of u eps}
            \begin{cases}
         L_{\gamma_j}v = 0 & \text{ in } \Omega_T,\\
        v =\eta & \text{ on } (\Omega_e)_T,\\
          v(0) = 0,\, \partial_{t}v(0) = 0 & \text{ on } \Omega
    \end{cases}
        \end{equation}
        and
        \begin{equation}
        \label{eq: nonlinear part of u eps}
              \begin{cases}
         L_{\gamma_j}R = -f_j(u_\eps) & \text{ in } \Omega_T,\\
        R =0 & \text{ on } (\Omega_e)_T,\\
          R(0) = 0,\, \partial_{t}R(0) = 0 & \text{ on } \Omega,
    \end{cases}
        \end{equation}
        respectively. This simply follows from the unique solvability of \eqref{eq: PDE for semilinear problem} and both functions $u_\eps$ and $\eps v_j+R^{(j)}_\eps$ are solutions. Furthermore, we notice that the energy estimate of \cite[Theorem 3.1]{Semilinear-nonlocal-wave-eq}, \cite[eq.~(3.18)]{Semilinear-nonlocal-wave-eq} and the $r+1$ homogeneity of $f_j$ ensure that $R^{(j)}_\eps$ satisfies
        \begin{equation}
        \label{eq: energy estimate remainder}
        \begin{split}
             \|\partial_t R^{(j)}_\eps\|_{L^{\infty}(0,T;L^2(\Omega))}+\|R^{(j)}_\eps (t)\|_{L^{\infty}(0,T;H^s(\R^n))}&\lesssim \|f_j(u_\eps)\|_{L^2(\Omega_T)}\\
             &\lesssim \|u_\eps\|^{r+1}_{L^{\infty}(0,T;H^s(\R^n))}.
        \end{split}
        \end{equation}
        Moreover, we may estimate
        \begin{equation}
        \label{eq: estimate u eps}
        \begin{split}
             &\|\partial_t u_\eps\|_{L^{\infty}(0,T;L^2(\R^n))}+\|u_\eps\|_{L^{\infty}(0,T;H^s(\R^n))}\\
             &\lesssim \|\partial_t (u_\eps-\eps\eta)\|_{L^{\infty}(0,T;L^2(\Omega))}+\|u_\eps-\eps\eta\|_{L^{\infty}(0,T;H^s(\R^n))}+\eps\|\eta\|_{W^{1,\infty}(0,T;H^{2s}(\R^n))}\\
             &\lesssim \eps\|\eta\|_{W^{1,\infty}(0,T;H^{2s}(\R^n))}.
        \end{split}
        \end{equation}
        This follows from the following observations. If $u$ solves  \eqref{eq: PDE for semilinear problem} for a damping coefficient $\gamma\in C^{0,\alpha}(\R^n)$, a weak nonlinearity $f$ and $\varphi\in C_c^{\infty}((\Omega_e)_T)$, then $v=u-\varphi$ solves
        \begin{equation}
         \begin{cases}
         L_{\gamma}v+f(v) = -(-\Delta)^s\varphi & \text{ in } \Omega_T,\\
        v =0& \text{ on } (\Omega_e)_T,\\
          v(0) = 0,\, \partial_{t}v(0) = 0 & \text{ on } \Omega.
         \end{cases}
         \end{equation}
        Now, we may invoke \cite[eq.~(3.15)]{Semilinear-nonlocal-wave-eq} to find that there holds
        \[
            \begin{split}
                &\|\partial_t v(t)\|_{L^2(\Omega)}^2+\|v(t)\|_{H^s(\R^n)}^2\\
                &\lesssim \int_0^t|\langle \gamma\partial_t v,\partial_t v\rangle_{L^2(\Omega)}|\,d\tau +\int_0^t|\langle (-\Delta)^{s}\varphi,\partial_t v\rangle_{L^2(\Omega)}|\,d\tau\\
                &\lesssim  \|(-\Delta)^{s}\varphi\|_{L^2(0,t;L^2(\Omega))}^2+\int_0^t\|\partial_t v\|_{L^2(\Omega)}^2\,d\tau.
            \end{split}
        \]
        Thus, Gronwall's inequality gives
        \[
         \|\partial_t v(t)\|_{L^2(\Omega)}+\|v(t)\|_{H^s(\R^n)}\lesssim  \|(-\Delta)^{s}\varphi\|_{L^2(0,t;L^2(\Omega))}.
        \]
        This ensures the validity of the second estimate in \eqref{eq: estimate u eps}.
        Next, observe that by subtracting the PDEs for $u^{(1)}_\eps$ and $u^{(2)}_\eps$, we deduce that
        \begin{equation}
        \label{eq: condition for uniqueness semilinear}
            (\gamma_1-\gamma_2)\partial_t u_\eps=f_2(u_\eps)-f_1(u_\eps)\text{ in }\Omega_T.
        \end{equation}
        By \eqref{eq: decomposition of u eps}, we may write
        \begin{equation}
        \label{eq: condition for uniqueness semilinear 2}
         (\gamma_1-\gamma_2)(\eps\partial_tv_1+\partial_t R^{(1)}_\eps) =f_2(u_\eps)-f_1(u_\eps)\text{ in }\Omega_T.
        \end{equation}
        Combining \eqref{eq: energy estimate remainder} and \eqref{eq: estimate u eps}, we see that
        \begin{equation}
        \label{eq: decay estimate R eps}
            \|\partial_t R^{(j)}_\eps\|_{L^{\infty}(0,T;L^2(\Omega))}+\|R^{(j)}_\eps (t)\|_{L^{\infty}(0,T;H^s(\R^n))}\lesssim \eps^{r+1}.
        \end{equation}
        Multiplying \label{eq: condition for uniqueness semilinear 2} by $\eps^{-1}$ gives
         \begin{equation}
        \label{eq: condition for uniqueness semilinear 3}
        \begin{split}
            &(\gamma_1-\gamma_2)(\partial_tv_1+\eps^{-1}\partial_t R^{(1)}_\eps) =f_2(\eps^{-1/(r+1)}u_\eps)-f_1(\eps^{-1/(r+1)}u_\eps)\text{ in }\Omega_T.
        \end{split}
         \end{equation}
         Next, let us focus one the case $2s<n$ as the other one can be treated similarly. As $r>0$ we deduce from \eqref{eq: estimate u eps} that $\eps^{-1/(1+r)}u_\eps\to 0$ in $L^{\infty}(0,T;H^s(\R^n))$ and so by Sobolev's embedding in $L^q(0,T;L^{2_s^*}(\Omega))$ for all $1\leq q\leq \infty$ and $2_s^*=\frac{2n}{n-2s}$. Hence, by our assumptions on $f_j$ and \cite[Lemma 3.6]{zimmermann2024calderon}, we get
         \begin{equation}
         \label{eq: zero convergence nonlinearity}
            f_j(\eps^{-1/(r+1)}u_\eps)\to 0\text{ in }L^{q/(r+1)}(0,T;L^{2_s^*/(r+1)}(\Omega))
         \end{equation}
         for all $q\geq r+1$ as $\eps\to 0$. Additionally, using \eqref{eq: decay estimate R eps} we know that
         \begin{equation}
         \label{eq: zero convergence time derivative remainder}
              \eps^{-1}\partial_t R^{(j)}_\eps\to 0\text{ in }L^{\infty}(0,T;L^2(\Omega)).
         \end{equation}
         Therefore, from \eqref{eq: condition for uniqueness semilinear 3}, \eqref{eq: zero convergence nonlinearity} and \eqref{eq: zero convergence time derivative remainder}, we infer
         \[
            (\gamma_1-\gamma_2)\partial_t v_1=0\text{ in }\Omega_T.
         \]
         In particular, this ensures that there holds
         \[
            \int_{\Omega_T}(\gamma_1-\gamma_2)\partial_t (v_1-\eta)(w_2-\psi)^\star\,dxdt=0
         \]
         for any $\psi \in C_c^{\infty}((W_2)_T)$, where $w_2$ is the unique solution of
         \begin{equation}
                 \begin{cases}
         L_{\gamma_2}w = 0 & \text{ in } \Omega_T,\\
        w =\psi & \text{ on } (\Omega_e)_T,\\
          w(0) = 0,\, \partial_{t}w(0) = 0 & \text{ on } \Omega.
    \end{cases}
         \end{equation}
         Now, arguing as in the previous section, we get $\gamma_1=\gamma_2$ in $\Omega$. Hence, \eqref{eq: condition for uniqueness semilinear} reduces to 
         \begin{equation}
              f_1(u_\eps)=f_2(u_\eps)\text{ in }\Omega_T.
         \end{equation}
         Multiplying this identity by $\eps^{-(r+1)}$ and arguing as before, we deduce that
         \[
            f_1(v)=f_2(v)\text{ in }\Omega_T,
         \]
         where $v\vcentcolon =v_1=v_2$ as $\gamma_1=\gamma_2$. One can now show $f_1(x,\tau)=f_2(x,\tau)$ for all $x\in\Omega$ and $\tau\in\R$ exactly as described in \cite[p.~29]{Optimal-Runge-nonlocal-wave}. Hence, we can conclude the proof.
    \end{proof}



	
	\medskip 
	
	\noindent\textbf{Acknowledgments.} 
        P.~Zimmermann was supported by the Swiss National Science Foundation (SNSF), under grant number 214500.

    \section*{Statements and Declarations}
	
	\subsection*{Data availability statement}
	No datasets were generated or analyzed during the current study.
	
	\subsection*{Conflict of Interests} Hereby we declare there are no conflict of interests.

	\bibliography{refs} 
	
	\bibliographystyle{alpha}
	
\end{document}